\documentclass[12pt]{article}
\usepackage{amssymb,amsmath,psfrag}

\usepackage{amsmath}
\usepackage{amsfonts}
\usepackage{amsthm}
\usepackage{amssymb}
\usepackage{amscd}
\newtheorem{definition}{Definition}
\newtheorem{proposition}{Proposition}

\newtheorem{corollary}{Corollary}
\newtheorem{lemma}{Lemma}
\newtheorem{theorem}{Theorem} 
\newtheorem{example}{Example}

\usepackage{graphicx}
\begin{document}

\title{Some Baumslag-Solitar groups are two bridges virtual knots}

\author{J. G. Rodr\'{i}guez \and O. P. Salazar-D\'iaz \and J. J. Mira
 \footnote{ Escuela de Matem\'aticas,
Universidad Nacional de Colombia
Medell\'in, Colombia
jjmira@unal.edu.co, jgrodrig@unal.edu.co, opsalazard@unal.edu.co}}

%\address{ Escuela de Matem\'aticas\\
%Universidad Nacional de Colombia\\
%Medell\'in, Colombia\\
%jjmira@unal.edu.co, jgrodrig@unal.edu.co, opsalazard@unal.edu.co }

\maketitle

\begin{abstract}
In this paper we give necessary conditions on group presentations, with two generators and one relator, in order to be the group of a virtual knot diagram. Although those conditions are not enough, we use them to determine, completely, whether or not a Baumslag-Solitar group is the group of a $2$-bridge virtual knot. Moreover, we present a combinatorial proof of the fact that these groups are not $2$-bridge classical knot groups.
\end{abstract}

%\keywords{Two bridge virtual knot groups, knot groups, Baumslag-Solitar groups, Residually finite groups, Two bridge classical knot groups}

%\ccode{Mathematics Subject Classification 2000: 20E26, 57M05,57M27,57Q45}

\section{Introduction}
The group of a virtual knot was introduced by Kauffman in \cite{Ka} as a generalization of the Wirtinger presentation of the fundamental group of a classical knot complement. An important question in such category is to determine when a group $G$ is the group of a virtual knot diagram. Several conditions have been given on group presentations, see \cite{Kim}, \cite{RoTo} and \cite{SiWi}, but all of them depend on whether we can or can not get a cyclic Wirtinger presentation of deficiency $0$ or $1$, for $G$. 

A \textit{Wirtinger presentation} of $G$ is a presentation of the form
\begin{center}
$G(p,w_{1},...,w_{q}):=\left\langle x_{1},x_{2},\cdots,x_{p} \mid r_{1},r_{2},\cdots,r_{q}\right\rangle$,
\end{center}
where each $r_{k}:=w_{k}x_{j}w_{k}^{-1}x_{i}^{-1}$, for some $1\leq i,j\leq p$\ and $w_{1},\cdots,w_{q}$\ are words in the free group $F(x_{1},...,x_{p})$, not necessarily different in $G$.

In particular, a Wirtinger presentation is called \textit{cyclic} if we have that $r_{i}=w_{i}x_{i}w_{i}^{-1}x_{i+1}^{-1}$, for each $i=1,2,...,q$. The \textit{deficiency} of a presentation $\left\langle x_{1},...,x_{p} \mid r_{1},...,r_{q}\right\rangle$ is defined as $\left|p-q\right|$.  More generally \textit{the deficiency} of a group $G$, denoted by $d(G)$, is the maximum among the deficiency of its presentations.

In \cite{RoTo} Rodr\'{i}guez and Toro proved that the group of a $n$ bridge virtual knot, has a cyclic Wirtinger presentation of deficiency $0$ or $1$. In this paper we will provide an algorithm to construct, from a cyclic Wirtinger presentation of deficiency $0$ or $1$, a virtual knot diagram. Such algorithm uses the definition of combinatorial knots, see \cite{To}. 

Baumslag-Solitar groups constitute an important family of counterexamples in combinatorial group theory, see \cite{BaSo}, \cite{Me}, \cite{Co} and \cite{CeCo}. They were first introduced by Baumslag and Solitar in \cite{BaSo}, in order to get no-Hopfian one-relator group presentations.  Furthermore, they are given by the following  general  presentation form 
\begin{center}
$BS(m,n)=\left\langle x,y \mid xy^{m}x^{-1}=y^{n}\right\rangle$
\end{center}
where $m$ and $n$ are non zero integer numbers.

Among many properties of these groups, we have those of being residually finite and hopfian under some restrictions. That is summarized in the following theorem. The proof of the statements can be found in \cite{Me} and \cite{AnRaVa} respectively.

\begin{theorem} 
     
(a)  The group $BS(m,n)$ is residually finite if and only if $\left|m\right|=\left|n\right|$ or $\left|m\right|=1$ or $\left|n\right|=1$.

(b) A Baumslag-Solitar group $BS(m,n)$ is Hopfian if and only if it is residually finite or $\pi(m)=\pi(n)$, where $\pi(m)$ stands for  the set of prime divisors of $m$. 

\end{theorem}

 In this paper we give necessary conditions on a group presentation of the form $\left\langle x,y \mid r\right\rangle$ to be a virtual knot group. These conditions do not depend on  the Wirtinger presentation form. Although these conditions are not enough,  we can determine, by using them, when a Baumslag-Solitar group is the group of a $2$ bridge virtual knot, moreover we prove that any of those groups can't be the group of a $2$ bridge classical knot.

This paper is organized as follows. In section $2$ we give a short introduction to virtual and combinatorial knots. In Section $3$ we recall Kauffman's definition of virtual knot groups and some properties of them. We give some results about when group presentations correspond to virtual knot groups. In Section $4$ we discuss a necessary condition on presentations of the form $\left\langle x,y \mid r\right\rangle$ to be the group of a virtual knot diagram. In Section $5$, some necessary and sufficient conditions on Baumslag-Solitar groups to be the group of a virtual knot diagram are described. In Section $6$ we present a short review about the presentation of the group of a two bridge  classical knot. Even though, it is well known that the group of a classical knot has deficiency 1, we give a combinatorial proof of that fact for the case of two bridge knots and we relate that to conditions on Baumslag-Solitar groups to show that they are not  two bridge classical knot groups. Also, we give a short introduction on Fox's derivations to classify the center of Baumslag-Solitar groups to conclude that none of them correspond to torus knot groups. A more general proof of the fact that no Baumslag-solitar groups can not be the group of a classical knot is obtained from \cite{Sh}.

\section{Virtual Knots}

A \textit{virtual knot diagram} $K$ on $S^{2}$ is an oriented $4-$valent planar connected graph, whose crossings are classified according to Figure \ref{f1}-$a$. An example of a virtual knot diagram is shown in Figure \ref{f1}-$b$.

\begin{figure}[ht]
\begin{center}
\includegraphics[scale=0.55]
{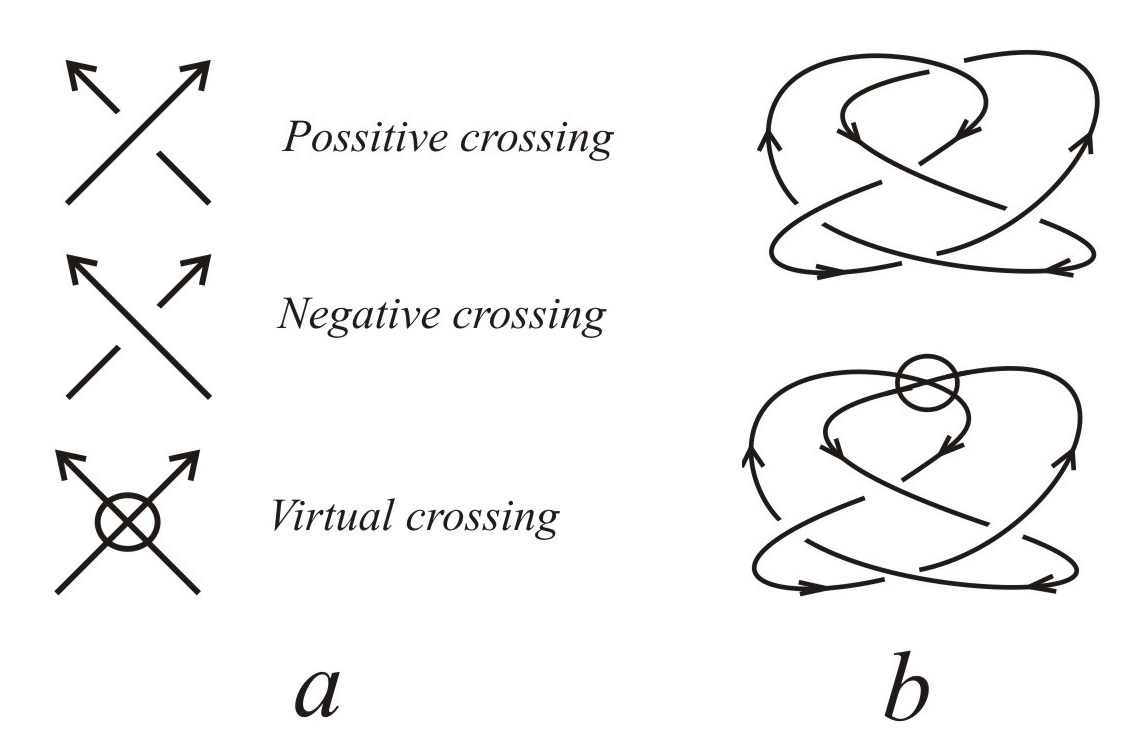}
\caption{(a) Crossing types. (b) Examples of virtual knots diagrams.}
\label{f1}
\end{center}
\end{figure}
Note that the positive and negative crossings have the visual effect of passing over and under the diagram. That doesn't happen with virtual crossings.\\
Positive and negative crossings are  called \textit{classical crossings}. 
\begin{definition}
Two virtual knot diagrams $K$ and $K'$ are said to be \textit{virtually equivalent} if one of them can be changed into the other by using a finite number of extended Reidemeister moves described in Figure 2. A \textit{virtual knot} is an equivalence class of virtual knot diagrams.

\begin{figure}[ht]
\begin{center}
\includegraphics[scale=1.55]
{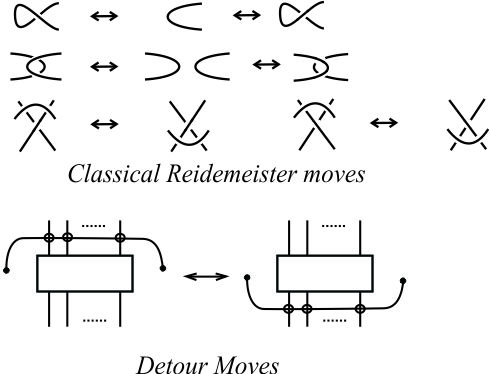}
\caption{Extended Reidemeister Moves}
\label{fig2}
\end{center}
\end{figure}

\end{definition}

A virtual knot diagram $K$ is called \textit{classical knot diagram} if $K$ does not have virtual crossings. A virtual knot is called \textit{classical knot} if its equivalence class has classical knot diagrams.

In \cite{GoPoVi} Goussarov, Polyak and Viro showed that the category of classical knots is contained in the virtual knots category. Kauffman (see\cite{Ka}) gave examples of virtual knots that are not equivalent to  classical knots. In this sense virtual knot theory becomes a non trivial extension of classical knot theory.

 The concept of  combinatorial knots, was introduced by M. Toro in \cite{To}, as equivalence classes of lists of the form $((i_{1},...,i_{2n}),(e_{1},...,e_{m}))$, where $\{i_{1},...,i_{2n}\}=\{\pm a_{1},...,\pm a_{n}\}$, $a_{j} \in \mathbb{Z}^{+}$, $e_{j}\in\left\{1,-1\right\}$, $j=1,2,...,n$ and $m=\max\left\{  a_{1},...,a_{n}\right\} $. Those lists are called  \textit{knot codes}. The \textit{trivial knot code} is the empty knot code $((),())$.   

The process of how to assign a knot code to a virtual knot diagram is illustrated with the following example. Let $K$ be the virtual knot diagram shown in Figure \ref{figura6}. We label its classical crossings with positive numbers and we choose $x$ over $K$, not a crossing point.

\begin{figure}[ht]
\begin{center}
\includegraphics[scale=0.85]{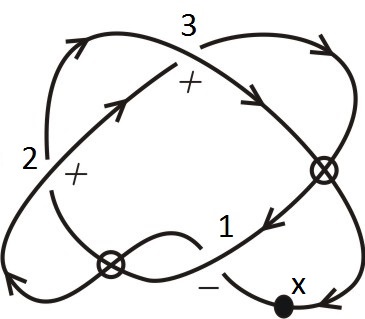}
\end{center}
\caption{Labeled virtual knot diagram}
\label{figura6}
\end{figure}

We start from $x$ at $K$ and write down the list of crossing labels we meet, with the convention that if we pass through an under crossing $a_{i}$, we add $-a_{i}$ to the list, but if we pass through an over crossing $a_{i}$, we add $a_{i}$. The resulting finite sequence is called a \textit{crossing list} of $K$. For our example, the crossing list is $(-1,2,-3,1,-2,3)$. We construct also the list $(e_{1}, e_{2}, e_{3})$, where, $e_{i} = 1$ if $i=a_{j}$ and $a_{j}$ is a positive crossing or $e_{i} = -1$ if $i=a_{j}$ and $a_{j}$ is a negative crossing, according to the classification given in Figure 1(a). This list is called a \textit{list of signs} of $K$, for our example the list of signs is $(-1,1,1)$. Hence, the pair $((-1,2,-3,1,-2,3),(-1,1,1))$ is a knot code of $K$.

In the previous example we used $\{1, 2, 3\}$ as the labeling set of the classical crossings of $K$. We can use any list of integers for labeling the classical crossings of the diagram, in that case the crossing list is generated in the same way, but we have to take care about the list of signs, because its length will be the maximum of the numbers used. 

If we use the numbers $1$, $2$,...,$n$ for labeling the classical crossing of virtual knot diagrams, then we say that the knot code is in \textit{standard form}. 

For a virtual knot diagram $K$, we will use the same letter $K$ to denote a knot code constructed by using the method above. The notation $[K]$ will represent the class of $K$ and it is called the combinatorial knot associated to $K$.

The concept of a combinatorial knot is related to that of a virtual knot. The specific relation is given in the following theorem, whose proof can be found in  \cite{Ro}.
\begin{theorem}
There exists a biunivocal correspondence between combinatorial knots and virtual knots. 
\end{theorem}
It is also possible to relate a knot code with classical knots.
\begin{definition}
A knot code $K$ is said to be \textbf{geometric} if there exists a classical knot diagram $K$ such that $K$ is a knot code for it. A combinatorial knot $[K]$ is \textbf{classical} if it has a geometric knot code. 
\end{definition}

Due to the biunivocal correspondence  between combinatorial and virtual knots, from now on, we will not make difference between them as we will use knot code and virtual knot diagram indistinctively.

\section{The virtual knot group}  
In this section we explain how to construct the group of a virtual knot and we describe some of its properties, the definitions and statements that are not proven here, can be found in \cite{RoTo}.

In the literature, there exists a more general definition for this group. Here we present one, based on what we will soon call the standard normal form of a diagram.

Let $K$ be a virtual knot diagram representative with $n$ classical crossings. We can choose  a point over the knot such that the first crossing in the path starting at it is an under crossing, and label that crossing $1$. Continue the way along the diagram until you cross it under another crossing, name it $2$. Following this way, all the $n$ crossings will be labeled.

It is easy to see that the associated knot code to the diagram $K$, can be written as $((-1,A_1,-2, A_2,...,-n, A_n),(e_1,...,e_n))$, where the sublists $A_1,...,A_n$ form a partition of the labeling set. Such knot code is called   \textit{standard normal form}.

We define the arcs of $K$ as the sublists $S_i=(-i, A_i, -(i+1))$ and $S_n= (-n, A_n, -1)$.
Since the definition of the group of a virtual knot  does not depend on representative elements, we use the standard normal form to introduce  the knot group of a virtual knot.

\begin{definition}\label{group}
Let $K=((-1,A_1,-2, A_2,...,-n, A_n),\, (e_1,...,e_n))$
 be a non trivial knot code and let $S_{1}$,...,$S_{n}$  be its arcs. We define the group of $K$ as 
\begin{equation}
\pi(K)=\left\langle S_{1},S_{2},\cdots,S_{n} \mid r_{1},r_{2},\cdots ,r_{n}\right\rangle, 
\end{equation}
where, for each $j=1,...,n-1$ 
\begin{equation}
r_{j}=S_{t_{j}}^{-e_j}S_{j}S_{t_{j}}^{e_j}S_{j+1}^{-1}, \ \ \ r_{n}=S_{t_{n}}^{-e_1}S_{n}S_{t_{n}}^{e_1}S_{1}^{-1}, 
\end{equation}
and $t_j$ is the unique index such that $j$ is in the list  $A_{t_{j}}$.

We define the group of the trivial knot code as $\mathbb{Z}$. 
\end{definition}
Note that the definition given above proves that the group of a virtual knot has a realizable Wirtinger presentation, it is clear that the reciprocal is also true.
\begin{example}

Consider the knot $K$ given in Figure \ref{figura7}.

\begin{figure}[h]
\begin{center}
\includegraphics[scale=0.4]{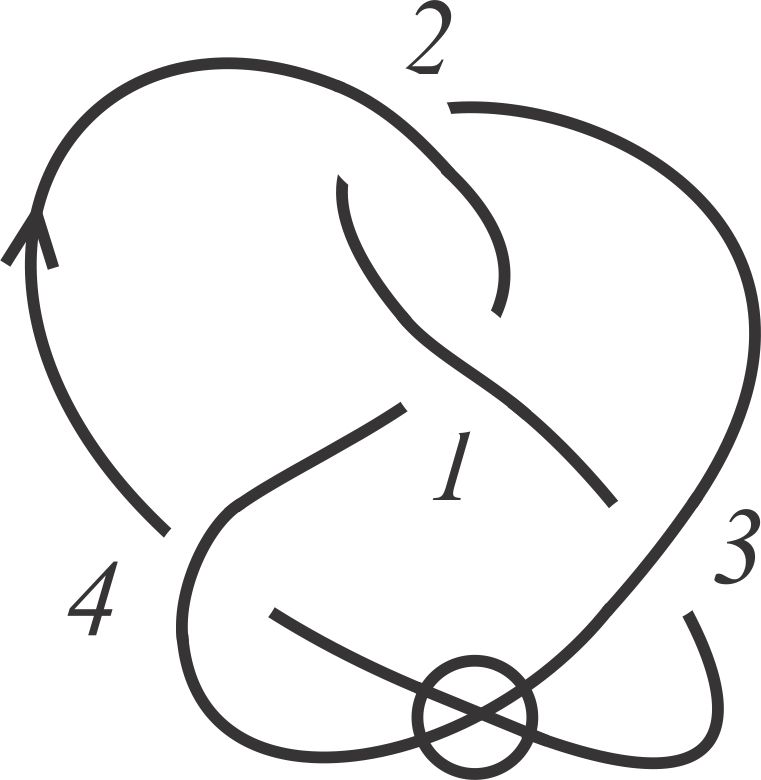}
\end{center}
\caption{Example arcs}
\label{figura7}
\end{figure}

We see that  a standard normal form for $K$ is 
$$((-1,4,3,-2,1,-3,-4,2),\, (-1,-1,-1,-1)).$$ 
Its arcs are
\begin{center}
 $S_1=(-1,4,3,-2)$, $S_2=(-2,1,-3)$, $S_3=(-3,-4)$, $S_4=(-4,2,-1)$
 \end{center}
 
 According to the definition just given, a presentation for the group of $K$ is
 $$\pi(K)=\left\langle S_{1},S_{2},S_{3}, S_4 \mid r_{1},r_{2},r_3 ,r_{4}\right\rangle$$ where\\ $r_1=S_2^{-1}S_1S_2S_2^{-1}=S_2^{-1}S_1$, $r_2=S_4^{-1}S_2S_4S_3^{-1}$, $r_3=S_1^{-1}S_3S_1S_4^{-1}$ and
$r_4=S_1^{-1}S_4S_1S_1^{-1}=S_1^{-1}S_4$.

\end{example}

A first property about this kind of groups is given in the following lemma.
\begin{lemma}
Let $K$ be a virtual knot. If $\pi=\pi(K)$ then, $\pi_{ab}=\left\langle S_{j}\right\rangle \cong \mathbb{Z}$ for each $j=1,2,...,n$.
\end{lemma}

Let $K =((i_{1},i_{2},\cdots ,i_{2n})),(e_{1},e_{2},\cdots ,e_{n})$ be a non trivial knot code and let $S_{1}$, $S_{2}$,...,$S_{n}$ be the arcs of $K$. We denote by $\left\vert S_{i}\right\vert $ the \textit{length} of $S_{i}$. 

\begin{definition} We say that $S_{i}$ is a \textbf{bridge} of $K$ if $\left\vert S_{i}\right\vert >2$. For a knot code $K$ the number of bridges is called the\textbf{\ bridge number}, and it is denoted by $br(K )$. For a combinatorial knot $[K]$ its bridge number is
defined by $br([K])=\min \{br(\tilde{K} ):[\tilde{K} ]=[K]\}$.
\end{definition}

Let $K=((i_{1},...,i_{2n}),(e_{1},...,e_{n}))$ be a knot code. Without loss of generality we may assume that the crossing list of $K$ is given by 
\begin{equation}
(A_{1},-1,...,-a_{1},A_{2},-(a_{1}+1),...,-a_{2}, A_{3},...,-a_{n-1},A_{n},-(a_{n-1}+1),...,-a_{n}),
\end{equation}
where $a_{i}>0$, $i=1,2,...,a_{n}$, $a_{0}=1$ and there is no crossing between $-(a_{t-1}+i)$ and $-(a_{t-1}+i+1)$, for $i=1,2,...,p=a_{t}-a_{t-1}$. 

Therefore, $y_{1}=(-a_{n},A_{1},-1)$,..., $y_{n}=(-a_{n-1},A_{n},-(a_{n-1}+1))$ will be the bridges of $K$.

For each $t=1,...,n$, we construct the word
\begin{equation*}
r_{t}=w_{t}^{-1}y_{t}w_{t}y_{t+1}^{-1}\text{, \qquad}w_{t}=y_{b_{1}}^{e_{a_{t-1}+1}}y_{b_{2}}^{e_{a_{t-1}+2}}...y_{b_{p}}^{e_{a_{t}}}\text{, }
\end{equation*}
where $y_{b_{1}},y_{b_{2}},...,y_{b_{p}}$ are bridges of $K$ such that $(a_{t-1}+i)$ is in the list ${A_{b_{i}}}$, $i=1,2,...,p=a_{t}-a_{t-1}$.

The following theorem provides another presentation for $\pi(K)$ that is called \textit{the over presentation.} For the proof of this result, see \cite{RoTo}.

\begin{theorem} 
\label{cy} Let $K$ be a non trivial knot code of $n$ bridges, and let $y_{1} $, $y_{2}$,...,$y_{n}$ be the bridges of $K$, then 
\begin{equation*}
\pi(K)\cong\left\langle y_{1},y_{2},...,y_{n} \mid r_{1},...,r_{n}\right\rangle 
\text{.}
\end{equation*}
\end{theorem}

\begin{definition}
Let $K=((i_{1},...,i_{2n}),(e_{1},....,e_{n}))$ be a non trivial knot code, and
\begin{center}
$\pi(K)\cong\left\langle y_{1},y_{2},...,y_{n} \mid w_{1}^{-1}y_{1}w_{1}y_{2}^{-1},...,w_{n-1}^{-1}y_{n-1}w_{n-1}y_{n}^{-1},w_{n}^{-1}y_{n}w_{n}y_{1}^{-1}\right\rangle $.
\end{center}
We define a \textbf{longitude} of $K$ as the word $l=w_{1}w_{2}\cdots w_{n}m^{-p}$, where the number $p=\sum^{n}_{j=1}{e_{j}}$ and  $m\in \{y_{1},...,y_{n}\}$. It is  denoted by $l=l(K)$.\\
We say that $m=m(K)$ is a meridian of $K$. The ordered pair $(l,m)$ is called a peripheral pair of $K$.
\end{definition}

It is not hard to verify that $l$ belongs to $[\pi(K),\pi(K)]$.
 
\begin{definition}
We say that two peripheral pairs $(l,m)$ and $(l',m')$ are \textbf{conjugate} if and only if there exists $w\in \pi(K)$ such that $l'=w^{-1}lw$, and $m'=w^{-1}mw$.
\end{definition}

Let $(l,m)$ be a peripheral pair of $K$, its conjugacy class $[(l,m)]$ will be called \textit{peripheral structure} of $K$.

\begin{lemma}
Let $K$ be a knot code, then the peripheral structure of $K$ is unique up to conjugation.
\end{lemma}

Let $K$ be a geometric knot code, if $[(l,m)]=[(e,m)]$, then $[K]=[(),()]$, see \cite{Wa}.

Now we present the reciprocal of  Theorem \ref{cy}. There is a proof of it in \cite{Kim}, however, we present a combinatorial one.

\begin{theorem}
Let 
\begin{equation}
G=\left\langle x_{1},x_{2},\cdots ,x_{n} \mid r_{1},r_{2},\cdots,r_{n}\right\rangle \text{, }
\label{w1}
\end{equation}
where $r_{i}=w^{-1}_{i}x_{i}w_{i}x_{i+1}^{-1}$, $i=1,2,...,n$ and $w_{1},\cdots,w_{n}$ are words in the free group $F(x_{1},...,x_{n})$, not necessarily different in $G$. Then there exists a knot code $\sigma$, of $n$ bridges, such that $\pi(\sigma)\cong G$.
\end{theorem}

\begin{proof}
Let $G$ be a group with presentation given by (\ref{w1}). We will suppose, for each $i=1,...,n$, that 
\begin{center}
$w_{i}=(x^{i_{1}}_{1}x^{i_{2}}_{2}\cdots x^{i_{n}}_{n})( x^{i_{n+1}}_{1}x^{i_{n+2}}_{2}\cdots x^{i_{2n}}_{n})\cdots (x^{i_{kn+1}}_{1}x^{i_{kn+2}}_{2}\cdots x^{i_{(k+1)n}}_{n})$,
\end{center}
where  $i_{j}\in \mathbb{Z}$, $j=1,2,...,(k+1)n$ and $k=k(i)$. 

Let $a_i:=l(w_i)$, the length of $w_i$, $c_0=0$ and $c_{t}=\sum _{i=1}^tl(w_i)=\sum _{j=1}^ta_j$.

Now, for each $s \in \{1,...,(k+1)n\}$ we take $f_{w_i}(s)=\sum^{s}_{j=1}\left|i_{j}\right|$.

By considering the following equation 
\begin{center}
$e_{c_{i-1}+f_{w_i}((j-1)n+t-1)+l}=sign(i_{(j-1)n+t}) \cdot 1$, 
\end{center}
$j \in \{1,2,...,k+1\}$, $t=1,2,...,n$, $l=1,2,...,f_{i}((j-1)n+t)-f_{i}((j-1)n+t-1)$ and $i=1,2,...,n$, we define the knot code $K$\begin{center}
${=((A_{1},-1,...,-c_1,A_{2},-({c_1+1}),...,-c_{2},...,-c_{n-1},A_{n},-({c_{n-1}+1}),...,-c_{n})}$,\\$(e_{1},...,e_{c_{n}}))$.
\end{center}
Where the elements in the list $A_t$ are:
\begin{center}
$\{A_{t}\}=\bigcup^{n}_{i=1} \bigcup^{k+1}_{j=1} \bigcup^{|i_t|}_{l=1} \{c_{i-1}+f_{w_i}((j-1)n+t-1)+l\}$, 
$t=1,2,...,n$. 
\end{center}
Thus, if we take $S_{t}=(-c_{t},A_{t},-(c_t+1))$, $t=1,2,...,n$, to be the bridges of $K$, then 
\begin{center}
$\pi(K) \cong \left\langle x_{1},...,x_{n} \mid r_{1},...,r_{n}\right\rangle$, 
\end{center} 
which by Definition \ref{group} satisfies $\pi(K) \cong G$.
\end{proof}

\begin{corollary}
Let $G$ be a group with presentation
 \begin{equation}
G=\left\langle x_{1},x_{2},\cdots ,x_{n} \mid r_{1},r_{2},\cdots,r_{n-1}\right\rangle \text{, }
\label{e}
\end{equation}
where for each $i=1,2,...,n-1$, $r_{i}=w^{-1}_{i}x_{i}w_{i}x_{i+1}^{-1}$ and  $w_{1},\cdots,w_{n-1}$ are words in the free group $F(x_{1},...,x_{n})$, not necessarily different in $G$. Then there exists a knot code $K$, of $n$ bridges, such that $\pi(K)\cong G$.
\end{corollary}
\begin{proof}
Let $w_{n}=(w_{1}w_{2}\cdots w_{n-1})^{-1}$, then
\begin{center}
$G=\left\langle x_{1},x_{2},\cdots ,x_{n} \mid r_{1},r_{2},\cdots,r_{n-1}, w^{-1}_{n}x_{n}w_{n}x^{-1}_{1}\right\rangle$ \text{, }
\end{center} 
and hence, there exists a $n$ bridges knot code $K$ such that $\pi(K)\cong G$.
\end{proof}

Note that the knot code $K$ constructed by the method described in the previous corollary has trivial longitude, so the combinatorial knot $[K]$ is not classical if and only if $G$ is not isomorphic to $\mathbb{Z}$. 

\begin{theorem} 
\label{t1} Let $G$ be a group with presentation 
\begin{equation}
G=\left\langle
x_{1},x_{2},\cdots,x_{m} \mid r_{1},r_{2},\cdots,r_{n}\right\rangle,
\label{w2}
\end{equation}
where $r_{k}=w^{-1}_{k}x_{j}w_{k}x_{i}^{-1}$, $1\leq i,j\leq m$\ and $%
w_{1},\cdots,w_{n}$\ are words in the free group $F(x_{1},...,x_{m})$, not
necessarily different in $G$. If the deficiency of $G$ is $0$ or $1$, and $G_{ab}\cong \mathbb{Z}$, then $G$ is the group of a combinatorial knot.
\end{theorem}

\begin{proof}
Suppose that $G$ is a group with presentation given by (\ref{w2}) and that $n=m-1$. Since $G_{ab}\cong \mathbb{Z}$, then $G$ has only two conjugacy classes. Let us denote them by $\{{e},H\}$. Therefore, for every $x\in \{x_{1},...,x_{m}\}$, in particular $x=x_{m}$, $H=[x]$. As a consequence, there are words $\omega_{1},...,\omega_{n-1}$ in $F(x_1,...,x_m)$ such that $G$ has the following presentation
\begin{center}
$G=\left\langle x_{1},...,x_{m} \mid \omega^{-1}_{1}x\omega_{1}x^{-1}_{1}=1,...,\omega^{-1}_{n-1}x\omega_{n-1}x^{-1}_{n-1}=1\right\rangle$.
\end{center}
Let us consider the following Tietze transformations on $G$:
\begin{equation}
\begin{array}{ccccc}
&x_{1}&=&\omega^{-1}_{1}x\omega_{1},\\
&x_{2}&=&\omega^{-1}_{2}(\omega_{1}x_{1}\omega^{-1}_{1})\omega_{2},\\
&x_{3}&=&\omega^{-1}_{3}(\omega_{2}x_{2}\omega^{-1}_{2})\omega_{3},\\
&&\vdots& \\
&x_{m-1}&=&\omega^{-1}_{m-1}(\omega_{m-2}x_{m-2}\omega^{-1}_{m-2})\omega_{m-1}.\\
\end{array}
\end{equation}
Then, $G$ is a combinatorial knot group.
\end{proof}

\section{Groups with two generators and one relator}

Let $F$ be the free group on $ X=\{x,y\}$. For a word $r \in F$ we define the \textit{length of $r$ with respect to $y$}, denoted by $l_{y}(r)$, as the sum of all the exponents of the letter $y$ in the word $r$; in a similar way we define $l_{x}(r)$.

\begin{theorem}
Let $G=\left\langle x,y \mid r\right\rangle$ be a group presentation where the relator is $r=x^{n_{1}}y^{m_{1}}x^{n_{2}}y^{m_{2}}\cdots x^{n_{k}}y^{m_{k}}$. If $l_{y}(r)=\pm 1$, then $G$ is the group of a virtual knot.
\end{theorem}

\begin{proof}
Suppose that $r=x^{n_{1}}y^{m_{1}}x^{n_{2}}y^{m_{2}}\cdots x^{n_{k}}y^{m_{k}}$, where $m_{j}, n_{j} \in \mathbb{Z}$, for every $j=1,2,...,k$. Let $w_{t}=y^{m_{t}}y^{m_{t+1}}\cdots y^{m_{k}}$, $t=1,2,...,k$. Then, from the relation $r(x,y)=1$ and the fact $y^{m_{t-1}}=w_{t-1}w^{-1}_{t}$, we have that 
\begin{center}
$\left(w^{-1}_{1}x^{n_{1}}w_{1}\right)\left(w^{-1}_{2}x^{n_{2}}w_{2}\right)\cdots \left(w^{-1}_{k}x^{n_{k}}w_{k}\right)=w^{-1}_{1}=y^{-l_{y}(r)}$.
\end{center}
Let us consider the following Tietze transformation on $G$: $y_{i}=w^{-1}_{i}x w_{i}$, $i=1,2,...,k$. Then $y^{n_{1}}_{1}y^{n_{2}}_{2}\cdots y^{n_{k}}_{k}=y^{-l_{y}(r)}$, and so
\begin{center}
$G=\left\langle y_{1},...,y_{k},y,x \mid r_{1},\cdots r_{k},y^{n_{1}}_{1}y^{n_{2}}_{2}\cdots y^{n_{k}}_{k}=y^{-l_{y}(r)}\right\rangle$.
\end{center} 
Where $r_{1}=w^{-1}_{1}xw_{1}y^{-1}_{1},...,r_{k}=w^{-1}_{k}xw_{k}y^{-1}_{k}$, with $w_t=y^{m_t+m_{t+1}+ \cdots +m_k}$ and $l_y(r)= \pm 1$. Note that $G$ has a Wirtinger presentation of deficiency 1. 

Due to the fact that $w^{-1}_{t}w_{t+1}=y^{-m_{t}}$, $t=0,1,...,k-1$, and the formulas given in (7), we obtain that $r_{1}=y^{-l_{y}(r)}xy^{l_{y}(r)}y^{-1}_{1},r_{2}=y^{-m_{2}}y_{1}y^{m_{2}}y^{-1}_{2}$,...,$r_{k}=y^{-m_{k}}y_{k-1}y^{m_{k}}y^{-1}_{k}$.
\end{proof}

The condition exposed in the previous theorem is not enough. For example, if $m$ and $l$ represent the meridian and longitude of the torus surface $T$, and $p$ and $q$ are co-prime numbers, then $K=pm+ql$ is a classical knot. $K$ is called \textit{toroidal knot} of type $(p,q)$ and it is denoted by $T(p,q)$. It is proven that the group of the knot $K$ has the presentation $G_{p,q}=\left\langle x,y \mid x^{p}=y^{q}\right\rangle$. See \cite{Ma}.

Although,  it is known that $\left\langle x,y \mid x^{p}=y^{q}\right\rangle _{ab}\cong \mathbb{Z}$, we present an algebraic proof of that fact.

\begin{proposition}
If $G=\left\langle x,y \mid x^{p}=y^{q}\right\rangle$, where $(p,q)=1$, then $G_{ab} \cong \mathbb{Z}$.
\end{proposition}
\begin{proof}
Suppose that $p$ and $q$ are co-prime numbers. Without loss of generality we may assume that $0<p<q$. Then there exist integer numbers $r_{1}$ and $s_{1}$, such that $1 \leq s_{1} <p$ and $q=r_{1}p+s_{1}$. So 
\begin{equation}
\begin{array}{cccc}
&G_{ab}&=&\left\langle x,y \mid y^{r_{1}p+s_{1}}=x^{p},[x,y]=1\right\rangle \\
&&=&\left\langle x,y \mid y^{s_{1}}=x^{p}y^{-r_{1}p},[x,y]=1\right\rangle \\
&&=&\left\langle x,y \mid y^{s_{1}}=(xy^{-r_{1}})^{p},[x,y]=1\right\rangle.
\end{array}
\end{equation}
By making $z=xy^{-r_{1}}$ we get  $G_{ab}=\left\langle y,z \mid y^{s_{1}}=z^{p},[y,z]=1\right\rangle$. So, if $s_{1}=1$, then $G_{ab} \cong \mathbb{Z}$. Suppose that $1<s_{1}$, and we have constructed a sequence of positive integers $1 \leq s_{k}<s_{k-1}<\cdots < s_{1}<p$, such that $G_{ab}=\left\langle x,y \mid x^{s_{i}}=y^{s_{i-1}},[x,y]=1\right\rangle$, for every $i=1,2,...,k$. If $s_{k}=1$, then $G_{ab} \cong \mathbb{Z}$. Suppose that $1<s_{k}<s_{k-1}$, then there are integer numbers $r_{k+1}$ and $s_{k+1}$ such that $1 \leq s_{k+1} < s_{k}$ and $s_{k-1}=r_{k+1}s_{k}+s_{k+1}$, therefore 
\begin{center}
$G_{ab}=\left\langle x,y \mid y^{r_{k+1}s_{k}+s_{k+1}}=x^{s_{k}},[x,y]=1\right\rangle=\left\langle x,y \mid y^{s_{k+1}}=x^{s_{k}}y^{-r_{k+1}s_{k}},[x,y]=1\right\rangle$.
\end{center}
So, if $w=xy^{-r_{k+1}}$, then $G_{ab}=\left\langle w,y \mid y^{s_{k+1}}=w^{s_{k}},[x,y]=1\right\rangle$. 

By recurrence, and the fact that $(p,q)=1$, we can construct a finite sequence $\{s_{k}\}^{t}_{k=1}$, that satisfies the condition above and $s_{t}=1$. Hence, $G_{ab}=\left\langle r,u \mid r=u^{s_{t-1}}\right\rangle \cong \mathbb{Z}$.
\end{proof}

\begin{corollary}
Let $G=\left\langle x,y \mid r\right\rangle$ be a group presentation. If $l_{y}(r)$ and $l_{x}(r)$ are relatively prime, then $G_{ab}\cong \mathbb{Z}$.
\end{corollary}
\begin{proof}
It follows from the fact that
$$G_{ab} \cong \left\langle x,y \mid x^{l_{x}(r)}=y^{-l_{y}(x)},[x,y]=1\right\rangle =\left\langle x,z \mid x^{l_{x}(r)}=z^{l_{y}(x)}\right\rangle _{ab}.$$
\end{proof}

\section{Baumslag-Solitar groups as virtual knot groups}

For two non-zero integers $m$ and $n$, the \textit{Baumslag-Solitar group}, denoted by $B(m,n)$, is the group given by the presentation
\begin{center}
$B(m,n)=\left\langle x,y \mid x^{-1}y^{m}x=y^{n}\right\rangle$.
\end{center}  

\begin{theorem}
The group $B(m,n)$ is the group of a virtual knot if and only if $n=m+1$.
\end{theorem}

\begin{proof}
If $B(m,n)$ is the group of a combinatorial knot then $B(m,n)_{ab}\cong\mathbb{Z}$, but $B(m,n)_{ab}=\left\langle x,y \mid x^{-1}y^{m}x=y^{n},[x,y]=1\right\rangle$, so we have that $B(m,n)_{ab}=\left\langle x,y \mid y^{m}=y^{n}\right\rangle$. Thus, $n=m+1$.

Conversely, 
\begin{center}
$B(m,m+1)=\left\langle x,y \mid x^{-1}y^{m}x=y^{m+1}\right\rangle=\left\langle x,y \mid y^{-m}x^{-1}y^{m}x=y\right\rangle$. 
\end{center}

If we change the generating set of $B(m,m+1)$ as follows:  $y_{1}=x$ and $y_{2}=y^{-m}xy^{m}$, then 
\begin{center}
$B(m,m+1)=\left\langle y_{1},y_{2} \mid y_{2}=(y^{-1}_{2}y_{1})^{-m}y_{1}(y^{-1}_{2}y_{1})^{m}\right\rangle$.
\end{center}
Therefore, $B(m,m+1)$ is the group of a $2$-bridge combinatorial knot. 
\end{proof}

An important result that relates Baumslag-Soliar groups and fundamental groups of connected and orientable three manifolds is resumed in the following theorem. Its proof is due to P. B. Shalen, see \cite{Sh}.

\begin{theorem}
If there are elements $x$ and $y$ in the fundamental group of a connected and orientable $3$-manifold and nonzero integers $m$ and $n$ such that $xy^{n}x^{-1}=y^{m}$. Then either $y$ has finite order, or $\left|n\right|=\left|m\right|$.
\end{theorem} 

Assume now that $K$ is a classical knot, such that the fundamental group of its complement is a Baumslag-Solitar group $B(m,n)$. Then, in view of Shalen's theorem,  $|n|=|m|$. Also, each classical knot group has infinite cyclic abelianization. Now, if we have a Baumslag-Solitar group $B(m,n)$ with $|n|=|m|$, then the abelianization of $B(m,n)$ cannot be infinite cyclic, unless $n=1$, $m=-1$ or vice-versa. The latter is the fundamental group of the Klein bottle which cannot be the fundamental group of a (classical) knot complement. Indeed, if $M$ is a compact irreducible $3$-manifold whose fundamental group is  isomorphic to the Klein bottle group, then $M$ is homeomorphic to an interval bundle over the Klein bottle. 

Hence, no knot complement group can be isomorphic to a Baumslag-Solitar group.

\section{Two bridges classical knots}

Although we have proved that no Baumslag-Solitar group is the fundamental group of a classical knot complement, in this section we want to present a combinatorial proof of the particular case that no Baumslag-Solitar groups correspond to fundamental groups of a $2$-bridge classical knots. In order to get that, we show that the deficiency of these type of groups is $1$, and from a result given by R. Riley \cite{Ri}, many of them are residually finite. 

It is well known that classical knot groups has deficiency $1$, but we want to present a combinatorial proof of that they are kmot groups, see \cite{Ri}, also we use Fox's derivation to show that no Baumslag-Solitar groups are not torus knot groups.
 
In \cite{Sc} Schubert describes and classifies the family of two bridge classical knots diagrams in terms of a normal form $S(\alpha,\beta)$, where $\alpha$ and $\beta$ are two odd and co-prime integers such that $0<\beta<\alpha$. 

For $k\in\{1,2,...,\alpha-1\}$, let $t_{k}$ be the integer in $(0,\alpha) \cup (\alpha, 2 \alpha)$ such that $k \beta \equiv t_{k}$ mod $2\alpha$. 

From Schubert's normal form and Theorem \ref{cy} we get the following theorem.

\begin{theorem}
\label{b2} The group of the $2$-bridge knot diagram  $S(\alpha ,\beta)$ has a presentation given by $G(2, y^{e_{1}}x^{e_{2}}\cdots y^{e_{\alpha-2}}x^{e_{\alpha-1}},x^{e_{1}}y^{e_{2}} \cdots x^{e_{\alpha-2}}y^{e_{\alpha-1}})$.
Where $e_{k}=-1$ if $t_{k}\in (\alpha ,2\alpha )$ and $e_{k}=1$ if $t_{k}\in (0,\alpha )$, for every $k=1,2,...,\alpha -1$.
\end{theorem}

The following lemma provides us another useful way to compute the numbers $e_{j}$, $j=1,2,...,\frac{\alpha-1}{2}$.

\begin{lemma}
For every $j\in \{1,2,...,\alpha -1\}$, let $\widehat{e}_{j}=sign(c_{j})$ where $-\alpha <c_{j}<\alpha $ and $j\beta \equiv c_{j}$ mod $2\alpha$.
Then $e_{j}=\widehat{e}_{j}$.
\end{lemma}
\begin{proof}
Let $j\in \{1,2,...,\alpha -1\}$ and consider the following cases:

Case $1$: $1\leq t_{j}<\alpha $, then $c_{j}=t_{j}$, so $\widehat{e}_{j}=1=e_{j}$.

Case $2$: $\alpha <t_{j}<2\alpha $, then $c_{j}=t_{j}-2\alpha $, and so $-\alpha <c_{j}<0<\alpha $ thus $\widehat{e}%
_{j}=-1=e_{j} $.
\end{proof}
An important property that  the exponents $e_{j}$ satisfy is 

\begin{corollary}
For every $j\in \{1,2,...,\frac{\alpha -1}{2}\}$, $e_{\alpha-j}=e_{j}$.
\end{corollary}

\begin{proof}
Let $j\in \{1,2,...,\frac{\alpha -1}{2}\}$, then $(\alpha -j)\beta\equiv (\alpha -j\beta )$ mod $2\alpha $, therefore 
\begin{equation}
t_{\alpha-j}=\left\{
\begin{array}{ccc}
\alpha -t_{j}& if & t_{\alpha -j} \in (0,\alpha )\\
&&\\
\alpha-c_{j} & if& t_{\alpha -j} \in (\alpha ,2\alpha )
\end{array}
\right.
\end{equation}
Now, $e_{\alpha -j}=1$ if and only if $1\leq t_{\alpha -j}<\alpha $ if and only if $1\leq \alpha -t_{j}<\alpha $ if and only if $\alpha
-1\geq t_{j}>0$ if and only if $1\leq t_{j}<\alpha $ if and only if $e_{j}=1$.

$e_{\alpha -j}=-1$ if and only if $\alpha <t_{\alpha -j}<2\alpha $ if and only if $\alpha <\alpha -c_{j}<2\alpha $ if and only if $0>c_{j}>\alpha $ if and only if $%
e_{j}=-1 $.
\end{proof}

Let $s=s(\alpha ,\beta )$ the only even number such that $2\leq s\leq \alpha -1$ and $s\beta \equiv 1$ mod $\alpha$ or $s\beta \equiv -1$ mod $\alpha $.

\begin{lemma}
\label{Lema(s)} Let $s$ be as above. Then 

(1) $e_{s}=1$, if $s\beta \equiv -1$ mod $\alpha$ and $e_{s}=-1$ if $s\beta \equiv 1$ mod $\alpha $.

(2) $e_{s+k}=-e_{k}$, for every $k=1,2,...,\alpha -s-1$ and

(3) $e_{s-k}=e_{k}$, for every $k=1,2,...,s-1$.
\end{lemma}

\begin{proof}
(1) Suppose that $s\beta \equiv -1$ mod $\alpha$, then there exists $t\in \mathbb{Z}$, such that $s\beta=-1+t\alpha $. Because $s\beta $ is an even number and $\alpha $ an odd number, $t$ must to an odd number, therefore $t=2m+1$. So, we have that $s\beta=-1+2m\alpha +\alpha =(\alpha -1)+m(2\alpha )$, and $s\beta \equiv(\alpha -1)$ mod $2\alpha $. Now, $1<\alpha -1<\alpha $, then $e_{s}=1$. 

On the other hand, let us suppose $s\beta \equiv 1$ mod $\alpha $. By definition of $s$, there is $t\in \mathbb{Z}$, such that $s\beta=1+t\alpha $. Because $s\beta $ is an even number and $\alpha $ is odd, $t $ must be odd, therefore $t=2m+1$. Then $s\beta=1+2m\alpha +\alpha =(\alpha +1)+m(2\alpha )$, as consequence $s\beta \equiv(\alpha +1)$mod$2\alpha $. Due to the fact that, $\alpha <\alpha +1<2\alpha $, $e_{s}=-1$.

(2) We prove the result in the case $s\beta \equiv -1$ mod $\alpha$. Let $k\in \{1,2,...,\alpha -s-1\}$ and suppose that $e_{k}=1$, therefore $0<t_{k}<\alpha $. Besides, $(s+k)\beta \equiv ((\alpha-1)+t_{k})$ mod $2\alpha $, and  because $0<t_{k}<\alpha $, we have $\alpha -1<\alpha -1+t_{k}<2\alpha -1$, so $t_{s+k}=(\alpha -1)+t_{k}$. Now, $(s+k)\in \{1,2,...,\alpha -1\}$ thus $(s+k)\beta $ is not congruent to $\alpha $ modulo $2\alpha $, from where $\alpha<t_{s+k}<2\alpha -1$ and so, $e_{s+k}=-1=-e_{k}$.

Suppose that $e_{k}=-1$, then $-\alpha <c_{k}<0$. Thus, $(s+k)\beta \equiv ((\alpha -1)+c_{k})$ mod $2\alpha $, and because $-\alpha <c_{k}<0$, we obtain that $-1<\alpha -1+c_{k}<\alpha -1$. Hence, $c_{s+k}=(\alpha -1)+c_{k}$. Due to the fact that $c_{s+k}\neq 0$, we can conclude that $0<c_{s+k}<\alpha -1$ and so $e_{s+k}=1=-e_{k}$.

The proof of (3) when $s\beta \equiv -1$ mod $\alpha$ and the proof of (2) and (3)  in the case $s\beta \equiv 1$ mod $\alpha$ goes very similarly.
\end{proof}

The proof of the next theorem is motivated by ideas presented in \cite{BrHi}.

\begin{theorem}
For every $\alpha $ and $\beta $, as previously, 
$$\pi(S(\alpha,\beta))=G(2,y^{e_{1}}x^{e_{2}}\cdots y^{e_{2}}x^{e_{1}}).$$
\end{theorem}

\begin{proof}
Let us denote $w_1=y^{e_1}x^{e_2}\cdots y^{e_{\alpha-2}}x^{e_{\alpha-1}}$ and  $w_2=x^{e_1}y^{e_2}\cdots x^{e_{\alpha-2}}y^{e_{\alpha-1}}$ and the relators $r_1: xw_1=w_1y$ and $r_2:yw_2=w_2x$.

Let's consider the case $s\beta \equiv -1 \mod \alpha$. Then, the relator $r_1$ becomes

\noindent
$xy^{e_1}x^{e_2}\cdots y^{e_{s-1}}x^{e_s}y^{e_{s+1}}\cdots y^{e_{\alpha-2}}x^{e_{\alpha-1}}=y^{e_{\alpha-1}}x^{e_{\alpha-2}}\cdots y^{e_s}x^{e_{s-1} }\cdots y^{e_2}x^{e_1}y$.\\
Now denote  $g_s=xy^{e_1}x^{e_2}\cdots y^{e_{s-1}}$ and $h_s=x^{e_{s-1}}y^{e_{s-2}} \cdots y^{e_{2}}x^{e_{1}}y$ and  define the relator $r_3:=g_s^{-1}r_1h_s^{-1}$. Then, if we make 
\begin{center}
$A_s=y^{e_{s+1}}\cdots y^{e_{\alpha-2}}x^{e_{\alpha-1}}y^{-1}x^{-e_1}y^{-e_2}\cdots x^{-e_{s-1}}$,\\
$B_s= y^{-e_{s-1}}\cdots x^{-e_2}y^{-e_1}x^{-1}y^{e_{\alpha-1}}x^{e_{\alpha-2}} \cdots x^{e_{s+1}}$
\end{center}
the relator $r_3$ can be written as $x^{e_s}A_s=B_sy^{e_s}$, but $e_s=1$, so  that the relator $r_3$ becomes  $xA_s=B_sy$.\\
From Lemma 8, $A_s$ can be written in the following form

$\begin{array}{rl}
A_s=&y^{-e_1}\cdots y^{-e_{\alpha-s-2}}x^{-e_{\alpha-s-1}}y^{-e_{\alpha-s}}x^{-e_{s-1}}y^{-e_{s-2}}\cdots x^{-e_1}\\
=&y^{-e_1}\cdots y^{-e_{\alpha-s-2}}x^{-e_{\alpha-s-1}}y^{-e_{\alpha-s}}x^{-e_{\alpha-s+1}}y^{-e_{\alpha-s+2}}\cdots x^{-e_{\alpha-1}}
\end{array}$

In the same way, we prove that 
$$B_s=y^{-e_1}\cdots x^{-e_{s-2}}y^{-e_{s-1}}x^{-e_s}y^{-e_{s+1}}x^{-e_{s+2}}\cdots x^{-e_{\alpha-1}}.$$

Thus, $A_s=B_s=w_2^{-1}$, and hence $r_3$ is equal to $xw_2^{-1}=w_2^{-1}y$. Therefore, $w_2x=yw_2$ is a consequence of $r_3$.

For the case $s\beta \equiv 1 \mod \alpha$ we consider the relator $r_3:=h_s^{-1}r_1q_s^{-1}$, but in this case we write $r_1$ as:

\noindent
$xy^{e_{\alpha-1}}x^{e_{\alpha-2}}\cdots x^{e_{s+1}}y^{e_s}x^{e_{s-1}}\cdots y^{e_2}x^{e_1}=y^{e_1}x^{e_2}\cdots y^{e_{s-1}}x^{e_s}\cdots y^{e_{\alpha-2}}x^{e_{\alpha-1}}y$.\\
Now we take $g_s=y^{e_s}x^{e_{s-1}}\cdots y^{e_2}x^{e_1}$ and $h_s=y^{e_1}x^{e_2} \cdots y^{e_{s-1}}x^{e_s}$, then $r_3$ becomes $x^{-e_s}B_s=A_sy^{-e_s}$. But $e_s=-1$, then $r_3$ is the relator $xB_s=A_sy$, where 
\begin{center}
$A_s=y^{e_{s+1}}\cdots y^{e_{\alpha-2}}x^{e_{\alpha-1}}yx^{-e_1}y^{-e_2}\cdots x^{-e_{s-1}}$,\\
$B_s= y^{-e_{s-1}}\cdots x^{-e_2}y^{-e_1}xy^{e_{\alpha-1}}x^{e_{\alpha-2}} \cdots x^{e_{s+1}}$.
\end{center}
By a similar process as in  the previous case, we prove that $B_s=A_s=w_2^{-1}$, then $xw_2^{-1}=w_2^{-1}y$ and so $w_2x=yw_2$ is a consequence of $r_3$.
\end{proof}

\begin{theorem}
A $2$-bridge classical knot $S(\alpha,\beta)$ is a torus knot of the type $(2,\alpha)$ if and only if $\beta = \pm 1$. Therefore, the center of a $2$-bridge knot group $S(\alpha,\beta)$ is trivial if and only if $\beta \neq \pm 1$.
\end{theorem}

\begin{proof} Since $\pi(S(\alpha,1))=\pi(S(\alpha,-1))$, then we only consider the case $\beta = 1$.
\begin{eqnarray*}
\pi(S(\alpha,1)) &=&\left\langle x,y \mid xyxyx...yx=yxyx...yxy\right\rangle \\
&\cong &\left\langle x,y,z \mid xyxyx...yx=yxyx...yxy,z=yx\right\rangle \\
&\cong &\left\langle x,y,z \mid xz^{\frac{\alpha -1}{2}}=z^{\frac{\alpha -1}{2}%
}y,y^{-1}z=x\right\rangle \\
&\cong &\left\langle y,z \mid y^{-1}z^{\frac{\alpha +1}{2}}=z^{\frac{\alpha -1}{2}%
}y\right\rangle \\
&\cong &\left\langle y,z,h \mid y^{-1}z^{\frac{\alpha +1}{2}}=z^{\frac{\alpha -1}{%
2}}y,h=yz^{\frac{\alpha -1}{2}}\right\rangle \\
&\cong &\left\langle z,h \mid z^{\frac{\alpha +1}{2}}=h^{2}z^{-\frac{\alpha -1}{2}%
}\right\rangle \\
&\cong &\left\langle z,h \mid z^{\alpha }=h^{2}\right\rangle = \pi(T(2,\alpha))\text{.}
\end{eqnarray*}
The rest of the proof is obtained by using the fact that prime knots are completely characterized by their fundamental groups up to inverses and mirror images.
\end{proof}

In order to get our main result we introduce some representation theory ideas.

Consider the following two matrices 
\begin{center}
$B_{t}=\left[\begin{array}{cc}
t&1\\
0&1
\end{array}\right]$, and $A_{t,u}=\left[\begin{array}{cc}
t&0\\
-tu&1
\end{array}\right]$ 
\end{center}
Let $\Gamma_{t,u}$ be the subgroup of $GL(2,\mathbb{C})$ generated by them. The subgroup $\Gamma_{t,u}$ has been studied in \cite{Ri}, and, for the particular case $t=1$, in \cite{PoTo}. Let $F$ be the free group on the set ${x,y}$, and let $H_{t,u}:F \rightarrow GL(2,\mathbb{C})$ be the unique homomorphism such that $H_{t,u}(x)=A_{t,u}$ and $H_{t,u}(y)=B_{t}$. For $w \in F$, let $W=\left[\begin{array}{cc} w_{11}&w_{12}\\w_{21}&w_{22}\end{array}\right]=H_{t,u}(w)$.

\begin{theorem}
Let $w=x^{e_{1}}y^{e_{2}}\cdots x^{e_{\alpha-2}}y^{e_{\alpha-1}}$, where $\alpha$ is an odd number and $e_{\alpha-j}=e_{j}$, $j=1,2,...,\frac{\alpha-1}{2}$. Then $H_{t,u}$ defines a homomorphism from $G(2,w)$ into $GL(2,\mathbb{C})$ if and only if 
\begin{center}
$w_{11}+(1-t)w_{12}=0$.
\end{center}
\end{theorem}
\begin{proof}
See \cite{Ri}.
\end{proof}

The \textit{Nab-rep polynomial} is defined in \cite{Ri} as the polynomial 
\begin{center}
$\Phi(t,u)=t^{-m}\left[w_{11}+(1-t)w_{12}\right]\in \mathbb{Z}[t,t^{-1},n]$.
\end{center}
Consider the notation $\Lambda=\mathbb{Z}[t,t^{-1}]$, therefore $\Phi(t,u)\in \Lambda[u]$ and we have the following result whose proof we omit but it can be found in \cite{Ri}.
\begin{theorem}
Let $G(2,w)$ be the group of a $2$-bridge classical knot, and let $\Phi$ its respective Nab-rep polynomial. Then there exists a monic factor, $\Phi_{1}$, of $\Phi$ in $\Lambda[u]$ such that the kernel of a generic representation for $\Phi_{1}$ is the center of $G(2,w)$.
\end{theorem}

A direct consequence of the previous theorem is that for every non toroidal $2$ bridge classical knot $K$, there exist $u,t$ in $\mathbb{C}$, such that $\pi(K)$ is isomorphic to $\Gamma_{t,u}$. 

To get to an important result about the Baumslag-Solitar groups relating these ideas to knot groups, let's recall the definition of residually finite groups. 
\begin{definition}
A group $G$ is said to be residually finite if for every $g\in G$, $g\neq 1$ there exists a homomorphism $\varphi :G \rightarrow G^{*}$, where $G^{*}$ is a finite group, such that $\varphi(g) \neq 1$. 
\end{definition}
The following theorem provides us an infinite and important family of residually finite groups. Its proof can be found in \cite{Ma}.   
\begin{theorem}
Let $R$ be a field and let $M$ be a finite set of $n \times n$ matrices with entries in R and with non-vanishing determinants. Then the matrices in $M$ generate a residually finite group.
\end{theorem}

So we have that for every $u,t$ in $\mathbb{C}$, $\Gamma_{t,u}$ is residually finite. Since the only classical knots with corresponding groups having non trivial center are the torus knots, we have the result.

\begin{theorem}
Every group of a $2$-bridge knot $S(\alpha,\beta)$, with $\beta \neq \pm 1$, is residually finite.

As a consequence, no Baumslag-Solitar group $BS(n,n+1)$ is the group of a $2$-bridge classical knot $S(\alpha,\beta)$, with $\beta \neq \pm 1$, if $n \ne \pm 1$.
\end{theorem}

Let $F$  be the free group on the set $X=\{x_{1},x_{2},\cdots ,x_{n}\}$, and let $\mathbb{Z}F$ denote the ring group associated to $F$. For every generator $x_{j}$ of $F$ let $\frac{\partial }{\partial x_{j}}:\mathbb{Z}F\rightarrow \mathbb{Z}F$ be the homomorphism such that $\frac{\partial }{\partial x_{j}}\left( x_{i}\right) =\delta _{ij},$ where $\delta _{ij}$ stands for the Kronecker delta function. 

Let $\theta :\mathbb{Z}F\rightarrow \mathbb{Z}$ be the $\mathbb{Z}$-homomorphism such that $\theta \left( g\right) =1$ for each $g\in F$. Each $\frac{\partial }{\partial x_{j}}$ is a \textit{derivation}, see \cite{Fox}, so they satisfy the following conditions:
\begin{enumerate}
\item $\frac{\partial }{\partial x_{j}} \left( uv\right) =\frac{\partial }{\partial x_{j}}\left( u\right) \theta \left( v\right) 
+u \frac{\partial }{\partial x_{j}} \left( v\right)$ \,\, for every $u,v$ in $\mathbb{Z}F$

\item $\frac{\partial }{\partial x_{j}} \left( n\right) =0$ for each $n\in \mathbb{Z}.$
\end{enumerate}

We have the following proposition. Its proof is a direct consequence of the definition of derivation, so we will omit it. 
\begin{proposition}
Let $u_{0}x_{j}^{p_{1}}u_{1}x_{j}^{p_{2}}u_{2}\cdots u_{q-1}x_{j}^{p_{q}}u_{q}$  be a word in the free group $F$, where $u_{0},\cdots ,u_{q}$ are words in $%
\{x_{1},x_{2},\cdots ,x_{j-1},x_{j+1},\cdots ,x_{n}\}.$ 
\begin{enumerate}
\item $\frac{\partial }{\partial x_{j}}\left( x_{j}^{p}\right) =\frac{%
x_{j}^{p}-1}{x_{j}-1}$, if $p\geq 1$.

\item $\frac{\partial }{\partial x_{j}}\left( x_{j}^{-p}\right)=-x_{j}^{-p}\left( \frac{x_{j}^{p}-1}{%
x_{j}-1}\right)$, if $p\geq 1.$

\item $\frac{\partial }{\partial x_{j}}\left(
u_{0}x_{j}^{p_{1}}u_{1}x_{j}^{p_{2}}u_{2}\cdots
u_{q-1}x_{j}^{p_{q}}u_{q}\right)
=\displaystyle \sum_{i=1}^{q}u_{0}x_{j}^{p_{1}}u_{1}x_{j}^{p_{2}}u_{2}\cdots u_{i-1}[\frac{%
x_{j}^{p_{i}}-1}{x_{j}-1}]$. 
\end{enumerate}

\end{proposition}

Let $G$  be a group given by a presentation of the form $G= \left\langle  x_{1},x_{2} \mid R \right\rangle$, and let $H$ be a normal subgroup of $G$ such that $G/H= \langle t \mid  \quad \rangle$. If $\pi :G\rightarrow G/H$ denotes the canonical homomorphism, then $\pi $ extends to a $\mathbb{Z}$-homomorphism $\pi :\mathbb{Z} G\rightarrow \mathbb{Z}\left( G/H\right)$. 

Let $I_{G}$ be the ideal of $\mathbb{Z}[t,t^{-1}]$ generated by the Laurent polynomials $\pi \left(\frac{\partial R}{\partial x_1}\right)$ and $\pi \left(\frac{\partial R}{\partial x_2}\right)$. Since $\mathbb{Z}[t,t^{-1}]$ is a principal ideal domain, then $I_{G}$ is generated by a polynomial called the \textit{Alexander polynomial of $G$ respect to $H$} and it is denoted by $\Delta _{H}\left( t\right)$. We know that $\Delta _{H}\left( t\right)$ is the great common divisor of $\pi \left(\frac{\partial R}{\partial x_1}\right)$ and $\pi \left(\frac{\partial R}{\partial x_2}\right)$.
 
\begin{theorem}
The polynomial $\Delta _{H}\left( t\right) $ is an invariant of $G$.
\end{theorem}

The proof of the following theorem can be found in \cite {Mu}.
 
\begin{theorem}
\label{the1}
Let $G$ be a group with a single defining relation. Suppose that the center of $G$ is not trivial. Then, for every normal subgroup $H$ such that $
G/H =\left\langle t \mid \quad \right\rangle$ is cyclic infinite, we have the following.

\begin{enumerate}
\item The degree of $\Delta _{H}\left( t\right)$, say $d$, is at least $1$.

\item If $d\geq 2,$ then $\Delta _{H}$ must divide $1-t^{r}$ for some positive integer $r$.

\item If $d=1$ then $\Delta _{H}\left( t\right) =n\left( 1\pm t\right) $, $n\in \mathbb{Z},$ $n\neq 0$.
\end{enumerate}

\end{theorem}

\begin{corollary}

For the group $G:=BS\left( m,n\right) $ 

\begin{enumerate}
\item If $m= n$ then the center of $G$ is the cyclic group $ \langle y^{n} \rangle$.

\item If $m\neq \pm n$, the center of $G$ is the trivial group.
\end{enumerate}

\end{corollary}

\begin{proof}

\bigskip 

We know that $BS\left( m,n\right) =\langle x,y \mid xy^{m}x^{-1}=y^{n}\rangle$. Let $%
R:=xy^{m}x^{-1}y^{-n}$ and $H:= \overline {y}= \langle x^{k}yx^{-k}:$ \ $k\in \mathbb{Z} \rangle$.
Then 
\[
\frac{\partial R}{\partial x}=1+xy^{m}\left( -x^{-1}\right) =1-xy^{m}x^{-1}%
\text{ and }
\]

\[
\begin{array}{cl}
\displaystyle \frac{\partial R}{\partial y} & =x\frac{y^{m}-1}{y-1}+xy^{m}x^{-1}\left(
-y^{-n}\left( 1+y+\cdots +y^{n-1}\right) \right)  \\ 
& =x\left( 1+y+\cdots +y^{n-1}\right) -xy^{m}x^{-1}y^{-n}\left( 1+y+\cdots
+y^{n-1}\right) .%
\end{array}%
\]
Since $\pi :\mathbb{Z}G\rightarrow \mathbb{Z}\left( G/H\right) $, we have that $\pi \left( x\right) =t$ and $\pi \left( y\right) =1$. Hence, $\pi \left( \frac{\partial R}{\partial x}\right) =0$ and $\pi \left( \frac{\partial R}{\partial y}\right) =mt-n$.

So, from  Theorem \ref{the1}, if $m\neq n,$ then the center of $BS\left( m,n\right) $ is trivial.\\
 Besides, if $m=\pm n$, then $y^{n}$ is in the center of $BS(n,n)$. Therefore, from Theorem $2$ of \cite{Mu}, because $y^{n}\neq 1$, we can conclude that $Z\left( BS\left( m,n\right)\right) =\langle y^{n} \rangle$.

\end{proof}

\begin{corollary}
The group $BS\left( m,m+1\right) $ is not isomorphic to a torus knot group.
\end{corollary}

\section*{Acknowledgments}
We want to thank professor Michael Kapovich for his suggestions of the proof that no knot complement group can be isomorphic to a Baumslag-Solitar group. Also we thank to the project $28792$ with QUIPU code:	$201010015036$.


\begin{thebibliography}{99}  \small  

\bibitem{AnRaVa} S. Andreadakis, E. Raptis, and D. Varsos, Residual finiteness and hopficity of certain HNN extensions. { \it Arch. Math.}, { \bf Vol. 47} (1986) 1--5 .  
\bibitem{BaSo} G. Baumslag, D. SolitarSome, two generator one-relator non-Hopfian groups. { \it Bull. Amer. Math. Soc}. { \bf Vol 689}  (1962) 199--201.
\bibitem{BuZi} G. Burde and H. Zieschang, Knots. {\it De Guyter Studies in Mathematics}. (1985).        
\bibitem{BrHi} G.W. Brumfield and H. M. Hilden, SL(2) Representations of Finitely Presented Groups. { \it Contemporary Math, AMS, Providence, United States.} { \bf Vol. 187} (1995).                                                                                
\bibitem{CeCo} T. Ceccherini-Silvertein and M. Coornaert, Cellular Automata and Groupd. { \it Monographs in Mathematics, Springer} (2010).
\bibitem{Co} D.J. Collins, On Recognizing Hopf Groups. { \it Archiv der Mathematik} {\bf 15. VII. Volume 20, Issue 3}, (1969) 235--240 .
\bibitem{Fox} R.H. Fox, Free differential calculus, I, II .  { \it Ann. Math}. {\bf 57} (1953) 547--560.
\bibitem{GoPoVi} M. Goussarov, M. Polyak and O. Viro, Finite Type Invariants of Classical and Virtual Knots. Preprint: math. GT/1981/9810073.
\bibitem{Ka} L. Kauffman, Virtual Knot theory. { \it Europ. J. Combinatorics},{ \bf Vol. 20} (1999), 663--691.
\bibitem{Kim} S. Kim, Virtual Knot Groups and their Peripheral Structure. Journal of Knot Theory and its Ramifications, \textbf{Vol. 9}, No. 6 (2000), 797-812.
\bibitem{Ma} A. Mal'cev,  On isomorphic matrix representations of infinite groups. {\it Mat. Sb}. {\bf 8 ,50} (1940), 405-422. (Russian) 
\bibitem{Me} S. Meskin, Non-residually Finite One-Relator Groups. {\it Trans. Amer.Math. Soc}. { \bf 64} (1972)  105--114. 
\bibitem{Mu} K. Murasugui, The center of a group with a single defining relation. {\it Math. Annales}. {\bf 155} (1964), 246--251.
\bibitem{PoTo} C. Pommerenke and M. Toro, On the two-parabolic subgroups of $SL(2, \mathbb{C})$. { \it Rev. Colombiana Mat}. {\bf 45, no. 1} (2011) 37--50.
\bibitem{Ri}  R. Riley, Non-abelian representations of 2-bridge knot groups. {\it Quart. J. Math. Oxford Ser}. {\bf (2) 35 , no. 138,} (1984) 191--208.
\bibitem{Ro} J. Rodr\'{\i}guez, Nudos Virtuales. Tesis Doctoral, Universidad Nacional de Colombia. 2011.
\bibitem{RoTo} J. Rodr\'{\i}guez and M. Toro, Virtual Knot Groups and Combinatorial Knots. {\it Sao Paulo Journal of Mathematical Sciences} {\bf 3,1} (2009), 297--314.
\bibitem{Sc} H. Schubert, Knoten mit zwei BrÃ¼cken. {\it Math}. {\bf  Z. 65}, (1956) 133--170.
\bibitem{Sh} P. B. Shalen, Three-manifolds and Baumslag–Solitar groups, {\it Topology and its Applications}. {\bf 110} (2001) 113--118.
\bibitem{SiWi} D.S. Silver and S.G. Williams,Virtual knot groups. Knots in Hellas '98, Proc. Int. Conf. Knot Theory Ramifications, eds. C. McA. Gordon (World Scientific, Singapore, 2000) 440--451.
\bibitem{To} M. Toro, Nudos Combinatorios y Mariposas. {\it Rev. Acad. Colomb. Cienc}. {\bf  28, 106} (2004), 79--86.
\bibitem{Wa} F. Waldhausen, On Irreducible 3-manifold which are sufficiently large. {\it Ann. of Math}. {\bf  87}, (1968)  56--88. 
\end{thebibliography}
\end{document}